\documentclass[12pt]{article}
\usepackage{amssymb,amsfonts,amsmath,amsthm,breqn,color}
\usepackage{caption,float,mathtools,url,cite,xcolor}
\newcommand{\one}[1]{\mbox {\bf 1}_{\{#1\}}}
\newcommand{\odin}{\mbox {\bf 1}}

\newcommand{\witi}{\widetilde}

\newcommand{\zz}{{\mathbb Z}}

\newcommand{\rr}{{\mathbb R}}
\newcommand{\cc}{{\mathbb C}}

\newcommand{\calj}{{\mathcal J}}

\newcommand{\cals}{{\mathcal S}}
\newcommand{\cali}{{\mathcal I}}

\newcommand{\caln}{{\mathcal N}}

\newcommand{\call}{{\mathcal L}}

\newcommand{\veps}{\varepsilon}

\newcommand{\beq}{\begin{eqnarray*}}
\newcommand{\feq}{\end{eqnarray*}}
\newcommand{\beqn}{\begin{eqnarray}}
\newcommand{\feqn}{\end{eqnarray}}

\newtheorem{theorem}{Theorem}
\newtheorem*{conj*}{Conjecture}
\makeatletter \@addtoreset{theorem}{section}\makeatother

\newcommand{\nn}{{\mathbb N}}
\makeatletter \@addtoreset{theorem}{section}\makeatother

\newtheorem{lemma}[theorem]{Lemma}

\newtheorem*{theorema*}{Theorem~A}
\newtheorem*{theoremb*}{Theorem~B}
\newtheorem*{cld*}{Condition $\mbox{LD}_d$}
\newtheorem*{theorem*}{Theorem}

\newtheorem{proposition}[theorem]{Proposition}
\newtheorem{corollary}[theorem]{Corollary}
\newtheorem{remark}[theorem]{Remark}
\newtheorem*{remark*}{Remark}
\newtheorem{example}[theorem]{Example}

\usepackage{algorithm}
\usepackage{graphicx}
\usepackage[noend]{algpseudocode}
\makeatletter
\def\BState{\State\hskip-\ALG@thistlm}
\makeatother

\DeclareCaptionLabelFormat{noname}{#2}
\newlength\myindent
\title{Staircase patterns in words: subsequences,
\\
subwords, and separation number}
\author{Toufik~Mansour\thanks{ Department of Mathematics, University of Haifa, 199 Abba Khoushy Ave, 3498838 Haifa, Israel;
\newline e-mail: tmansour@univ.haifa.ac.il}
\and
Reza~Rastegar\thanks{Occidental Petroleum Corporation, Houston, TX 77046 and Departments of Mathematics
and Engineering, University of Tulsa, OK 74104, USA - Adjunct Professor; e-mail:  reza\_rastegar2@oxy.com}
\and
Alexander~Roitershtein \thanks{Department of Statistics, Texas A\&M University, College Station, TX 77843, USA;
\newline e-mail: alexander@stat.tamu.edu}
}
\date{July 27, 2019}
\begin{document}
\maketitle
\begin{abstract}
We revisit staircases for words and prove several exact as well as asymptotic results for longest left-most staircase subsequences and subwords and
staircase separation number, the latter being defined as the number of consecutive maximal staircase subwords packed in a word. We study asymptotic properties of the sequence $h_{r,k}(n),$ the number of $n$-array words with $r$ separations over alphabet $[k]$ and show that for any $r\geq 0,$ the growth sequence $\big(h_{r,k}(n)\big)^{1/n}$ converges to a characterized limit, independent of $r.$ In addition, we study the asymptotic behavior of the random variable $\cals_k(n),$ the number of staircase separations in a random word in $[k]^n$ and obtain several limit theorems for the distribution of $\cals_k(n),$ including a law of large numbers, a central limit theorem, and the exact growth rate of the entropy of $\cals_k(n).$ Finally, we obtain similar results, including  growth limits, for longest $L$-staircase subwords and subsequences.
\end{abstract}
\section{Introduction and statement of results}
\label{intro}
Analysis of patterns in random words is one of the central topics in computer science, statistics, and combinatorics, with various important applications in biology, technology, and physics \cite{Bbook, analcombin, HM, Kbook}. In this paper, we are primarily concerned with the properties of a specific class of patterns called \textit{staircase} or \textit{smooth} patterns. Staircase patterns and related topics, such as $(a,b)$-rectangle patterns in permutations, have been the subject of several studies in recent years \cite{kitaev, mansour3,mansour2,mansour1}. Our initial motivation for studying these patterns stems out of their relation with a certain growth model of statistical physics \cite{depot}.
\par
We define \textit{words} as finite arrays of elements (\textit{letters}) from a linearly ordered set (\textit{alphabet})
$[k]:=\{1,\ldots,k\}$ for some given $k\in\nn.$ We refer to (not necessarily distinct) letters $i$ and $j$ in $[k]$ as \textit{neighbors} if $|i-j|\leq 1.$ Note that in the linear setup, in contrast to the cyclic one, the letters $1$ and $k$ are not considered to be neighbors. Let $\pi=\pi_1\pi_2\cdots \pi_n \in [k]^n$ be a $k$-ary word of the length $n.$ We say that $\pi$ is a \textit{staircase} if $\pi_{i+1}-\pi_i\in\{-1,0,1\}$ for all $i=1,2,\ldots,n-1$ \cite{mansour2} (such words are referred to as \textit{smooth} in \cite{mansour1,mansour3}). In other words, $\pi$ is a staircase if any pair of adjacent letters in the word are nearest neighbors in the
linear alphabet $[k].$ Staircase words can be interpreted as a certain class of Motzkin paths with steps $(1,1),$ $(1,-1)$ and $(1,0),$ laying in a strip with vertex heights bounded within the interval from $0$ to $k-1$ \cite{mansour2}.
\par
Our first object of interest is the \textit{staircase separation number}. For an integer $d\geq 1,$ we call a word $\pi$ \textit{$d$-separated} if it can be can be written as $\pi^{(1)}\pi^{(2)}\cdots\pi^{(d+1)}$ such that $\pi^{(i)}$s are maximal staircase subwords in $\pi,$ meaning while $\pi^{(i)}$ is a staircase, $\pi^{(i)}$ followed by the first letter in $\pi^{(i+1)}$ is not a staircase subword. Staircase words are considered $0$-separated. For a $d$-separated word $\pi\in [k]^n,$ we refer to $d$ as the staircase separation number of $\pi$ and denote it by $\cals_k(\pi).$  For instance, the word $12313242321$ is a concatenation of $5$ consecutive maximal staircase subwords $123, 1, 32, 4, 2321,$ and hence $\cals_k(12313242321) = 4.$ The frequency sequence $h_{r,k}(n)$ is then defined as the number of words $\pi$ in $[k]^n$ with $\cals_k(\pi)=r,$ that is
\beq
h_{r,k}(n) = \#\{ \pi \in [k]^n \,:\, \cals_k(\pi) = r \}.
\feq	
We denote by $\cals_k(n)$ the staircase separation number of a random word $\pi$ chosen uniformly from $[k]^n.$ The distribution of the random variable $\cals_k(n)$ is related to the sequences $h_{r,k}(n)$ through the identities
\beq
P^{(s)}_{k,n}(r):= P(\cals_k(n)=r) =\frac{1}{k^n}h_{r,k}(n),
\feq
where $P$ is a probability law that induces the uniform distribution on $[k]^n.$ It will be
technically convenient to define $P$ on the union $\cup_n [k]^n$ rather than on a single $[k]^n$ with a specific $n,$
and thus to have all random variables under consideration defined in the same probability space. Therefore,
we consider $P$ as the probability law of a given sequence of independent random variables $(w_n)_{n\in\nn},$ each distributed uniformly over the alphabet set $[k],$ and refer to $W_n=w_1\cdots w_n$ as a \textit{random word on $[k]^n.$} We record this for future reference:
\beqn
\label{random}
w_n, n\in\nn,~\mbox{\rm are independent},\quad P(w_i=w)=\frac{1}{k}~\forall\,w\in [k],\qquad W_n=w_1\cdots w_n.
\feqn
For example, in this notation, $\cals_k(n)=\cals_k(W_n).$
\par
Another object of interest is the longest $L$-staircase pattern, which we study in both subsequence and subword contexts. We say that $\pi'=\pi_{i_1}\pi_{i_2}\cdots\pi_{i_r}$ is the \textit{longest $L$-staircase} or \textit{left-most staircase} subsequence of $\pi$ if $\pi'$ is a staircase word such that $i_1=1,$ the differences $i_2-i_1,\ldots,i_r-i_{r-1}$ are all minimal, and $r$ is maximal under this rule. For a fixed $\pi\in [k]^n,$ we denote this maximal $r$ by $\calj_k(\pi).$ That is the sequence $(i_j)_{j\in\infty}$ is defined by setting $i_1=1$ and then, recursively,
\beq
i_j=\inf\big\{m>i_{j-1}:|\pi_m-\pi_{i_{j-1}}|\leq 1\big\},\qquad j\in\nn,
\feq
where, as usual, we assume that $\inf\emptyset=+\infty.$ Formally,
\beq
\calj_k(\pi) =\max\{r\in\nn:j_r<\infty\}.
\feq	
For instance, the $4$-ary word $141321$ is not a $L$-staircase word, however, its longest $L$-staircase subsequence is $1121$ with $\calj_4(141321) = 4.$
The frequency sequence $f_{r,k}(n)$ is then defined as the number of words $\pi$ in $[k]^n$ with $\calj_k(\pi)=r,$ that is
\beq
f_{r,k}(n) = \#\big\{ \pi \in [k]^n \,:\, \calj_k(\pi) = r \big\}.
\feq	
Furthermore, we define $\calj_k(n)$ as the length of the longest $L$-staircase subsequence of a random word distributed uniformly over $[k]^n.$ That is,
\beq
P^{ls}_{k,n}(r):= P(\calj_k(n)=r) =\frac{1}{k^n}f_{r,k}(n),
\feq
where $P$ is the uniform distribution on $[k]^n.$
\par
Similarly, we define the $L$-staircase subword or left-most staircase subword in a word $\pi\in [k]^n$ as $\pi'=\pi_1\pi_2\cdots\pi_r,$ where $\pi'$ is a staircase word and $r\leq n$ is maximal. For instance, the $4$-ary word $112141321$ is not a staircase word, its $L$-staircase subword is $1121.$
For $\pi\in[k]^n$ we set $\pi_0=\pi_1$ and define $\cali_k(\pi)$ as
\beq
\cali_k(\pi) = \max \big\{r\in\nn \,:\, r\leq n~\mbox{\rm and}~|\pi_i-\pi_{i-1}|\leq 1\, ~\forall\,i \leq r-1 \big\}.
\feq
We then define $g_{r,k}(n)$ as the number of words $\pi$ in $[k]^n$ such that $\cali_k(\pi)=r.$	That is,
\beq
g_{r,k}(n) = \#\big\{ \pi \in [k]^n \,:\, \cali_k(\pi) = r \big\}.
\feq	
In addition, we define $\cali_k(n)$ as the length of the longest $L$-staircase subword of a random word sampled uniformly from $[k]^n.$ Note that for all $r\in\nn_0,$
\beqn\label{Qnr}
P^{lw}_{k,n}(r):=P(\cali_k(n)=r)=\frac{1}{k^n}g_{r,k}(n).
\feqn
Given that $\cals_1(\pi)=\cals_2(\pi)=0$ and $\calj_1(\pi)=\calj_2(\pi)=\cali_1(\pi)=\cali_2(\pi)=n$ for all $\pi\in[k]^n$ and $n\in\nn,$
in what follows we assume that $k\geq 3.$
\par 	
To state our results we first introduce a few notations. Throughout this paper, $a_n\sim b_n,$ $a_n=O(b_n),$ and $a_n=o(b_n)$ for sequences $a_n$ and $b_n$ with elements that might depend on $r,k,$ and other parameters, means that, respectively, $\lim_{n\to\infty}\frac{a_n}{b_n}=1,$ $\limsup_{n\to\infty}\big|\frac{a_n}{b_n}\big|<\infty,$ and $\lim_{n\to\infty}\frac{a_n}{b_n}=0$ for all feasible values of the parameters when the latter are fixed. As usual, $a_n=\Theta(b_n)$ indicates that both $a_n=O(b_n)$ and $b_n=O(a_n)$ hold true. For a random variable $X,$ we denote its probability distribution by $\call(x),$ and its mean and variance by, respectively, $E(X)$ and $\sigma^2(X).$ We use the notation $\lim_{n\to\infty} \call(X_n)=X$ to indicate the convergence in distribution of a sequence of random variables $X_n$ to a random variable $X,$ as $n$ tends to infinity. We denote by $\caln(0,1)$ a standard normal random variable.
\par
Recall that the sequence of Chebyshev polynomials $U_n(t)$ of the second kind can be defined as the unique solution to the recursion equation
\beqn
\label{rc}
U_n(t)=2tU_{n-1}(t)-U_{n-2}(t), \quad n \geq 2,
\feqn
with initial conditions
\beqn
\label{rc1}
U_0(t)=1 \qquad \mbox{\rm and} \qquad U_1(t)=2t.
\feqn
The polynomials $U_n$ can be explicitly expressed as
\beqn
\label{expl}
U_n(t)=\frac{\big(t+\sqrt{t^2-1}\big)^{n+1}-\big(t-\sqrt{t^2-1}\big)^{n+1}}{2\sqrt{t^2-1}}.
\feqn
A standard reference for basic properties of Chebyshev's polynomial, which are frequently invoked in this paper, is \cite{chebyr}.
Notice the resemblance of \eqref{rc} to the Fibonacci recurrence and of \eqref{expl} to Binet's formula for Fibonacci numbers.
The intimate relation between Chebyshev polynomials and Fibonacci, Lucas, and more generally Horadam's numbers (i.\,e.,
solutions to a general second-order linear recursion with constant coefficients) has been intensively explored, among other areas, 
in enumerative combinatorics for computing generating functions (see, for instance, \cite{HM, mansour1, mansour3}),
number theory for representation of integer sequences (see, for instance, an early work \cite{recur} and recent \cite{aaa,sinverse1}),
and applied linear algebra for computing determinants and inverse matrices (see, for instance, \cite{ejc1,isumms,ejc,tan}).
The formula \eqref{rc} also provides a natural link between the theory of Chebyshev polynomials and
a discrete Poisson equation and its Sturm-Liouville theory \cite{ejc,ejc1,ejc5,ejc3,inverse3}, and subsequently to the harmonic functions
and first passage time of random walks (see, for instance, the classical \cite{kmg,spectral} and more recent \cite{dette,rwp,hargem}). 
Altogether, the link fundamentally bridges between analytic and probabilistic methods exploited in this paper.
\par 	
Our first result deals with the first and second moments of $\cals_k(n),$ $\calj_k(n)$ and $\cali_k(n).$
We remark that in the context of a ``cyclic order" alphabet, namely when $1$ and $k$ are considered as a neighbor pair, the above quantities are relatively easy to analyze. For instance, in the cyclic setup $\cali_k(n)$ becomes equal to a truncated geometric variable $\min\{G_k,n\},$
where $G_k$ is a  geometric random variable with parameter $\frac{k-3}{k},$ $\cals_k(n)$ is binomial $BIN\big(n-1,\frac{k-3}{k}\big),$ and
$\calj_k(n)$ is distributed as $1+BIN\big(n-1,\frac{3}{k}\big).$
The challenge in linear order alphabet $[k]$ is due to the fact that the letters $1$ and $k$
play a special role (in that they both have only two neighbors in the alphabet), and subsequently must be accounted with an extra care.
In the proof of the following theorem, even though asymptotic, as $n\to\infty,$ results for the moments can still be obtained using short probabilistic arguments, we choose an analytical approach, namely the analysis of generating functions, as a unified method of our analysis
in order to compute exact, that is for any given $n\in\nn,$ values of the moments.
\begin{theorem}
\label{Lstair-avg}
For any $k\geq 3,$
\begin{itemize}
\item [(i)] $E\big(\cals_k(n)\big) =\frac{(k-2)(k-1)}{k^2}(n-1),$
\item [(ii)] $\sigma^2\big(\cals_k(n)\big)= -\frac{4(k-2)}{k^4}+\frac{(k-2)(3k^2-5k+6)}{k^4}(n-1).$
\item [(iii)] $E\big(\cali_k(n)\big)  \sim\frac{2}{(k-3)^2U_k\big(\frac{k-1}{2}\big)}
+\frac{2U_{k-1}\big(\frac{k-1}{2}\big)}{(k-3)^2
U_k\big(\frac{k-1}{2}\big)} +\frac{(k^2-3k-2)}{(k-3)^2},$
provided that $k\geq 4.$ Furthermore, $E\big(\cali_3(n)\big)=5+o(1),$ as $n\to\infty.$
\item [(iv)] Provided that $k\geq 4,$ we have $\sigma^2\big(\cali_k(n)\big)=e_1(k)+e_2(k)-e_1^2(k)+o(1),$ where the constants $e_1(k)$ and $e_2(k)$
are explicitly defined below, in \eqref{e1} and \eqref{e2}, respectively. Furthermore, $\sigma^2\big(\cali_3(n)\big)=21+o(1),$ as $n\to\infty.$
\item [(v)] $E\big(\calj_k(n)\big)= 1+\frac{3k-2}{k^2}(n-1).$
\item [(vi)] $\sigma^2\big(\calj_k(n)\big)=\Big(a_k-\frac{(3k-2)(k^2-15k+10)}{k^4}\Big)n+O(1),$ where $a_k$ is defined by \eqref{a_n_lab}.
\end{itemize}
\end{theorem}
The proof of parts (i) and (ii) of the theorem is given in Section~\ref{aproof}, the proof of (iii) and (iv)
is included in Section~\ref{lolik}, and the claims in (v) and (vi) are proved in Section~\ref{prooft}.
It is instructive to compare the results in the above theorem with their analogue in the cyclic setup, when $1$
is formally identified with $k+1.$ For instance, for the expectations we have
\beq
E\big(\witi \cals_k(n)\big)=\frac{k-3}{k}(n-1),\quad E\big(\witi \cali_k(n)\big)=\frac{k}{k-3}\Big(1-\frac{1}{k^{n-1}}\Big),
\quad E\big(\witi \calj_k(n)\big)=1+\frac{3}{k}(n-1),
\feq
where $\witi \cals_k(n),$ $\witi \cali_k(n),$ and $\witi \calj_k(n)$  are counterparts of, respectively,
$\cals_k(n),$ $\cali_k(n),$ and $\calj_k(n)$ in the cyclic setup. Curiously enough, even though, in the linear alphabet case,
in contrast to the cyclic setup, the random variable $\calj_k(n)$ is not distributed the same as $n-\cals_k(n),$ we still have that $E\big(\cals_k(n)\big)+E\big(\calj_k(n)\big)=n.$
\par
Recall \eqref{random}. The base for our probabilistic analysis is the simple observation that the longest $L$-staircase of a random word
$W_n$ is a path of a nearest-neighbor symmetric random walk on $[k].$ In particular, the random variable $\cals_k(n)$ can be interpreted in terms of an additive functional of this random walk (see Section~\ref{mac3} and also Remark~\ref{hbor} below). A similar representation also holds for $\cali_k$ and $\calj_k$ (see Section~\ref{mac1} for the latter and the Appendix for the former). Using this observation we obtain the following laws of large numbers and central limit theorems for $\cals_k$ and $\calj_k,$
	
\begin{theorem}
\label{Lstair-llaws}
For any $ k\geq 3$ we have:	
\item [(i)] (LLN for $\cals_k$) $\psi_k:=\lim_{n\to \infty} \frac{\cals_k(n)}{n} = \frac{(k-2)(k-1)}{k^2},$ with probability one.
\item [(ii)] (CLT for $\cals_k$) There exists a constant $\sigma_k\in (0,\infty)$ such that
\beq
\lim_{n\to \infty} \call\bigg( \frac{\cals_k(n)-\psi_k n}{\sigma_k \sqrt{n}}\bigg)  = \caln(0,1).
\feq
\item [(iii)] (LLN for $\calj_k$) $\varphi_k:=\lim_{n\to \infty} \frac{\calj_k(n)}{n} = \frac{3k-2}{k^2},$ with probability one.
\item [(iv)] (CLT for $\calj_k$) There exists a constant $\nu_k\in (0,\infty)$ such that
\beq
\lim_{n\to \infty} \call\bigg( \frac{\calj_k(n)-\varphi_k n}{\nu_k \sqrt{n}}\bigg)  = \caln(0,1).
\feq
\end{theorem}
The proof of parts (i) and (ii) of the theorem is given in Section~\ref{mac3}. The proof of part (iii) and part (iv)
is similar, it is outlined in Section~\ref{mac1}. Some information on $\sigma_k^2$ and $\nu_k^2$ is available from the general theory
of Markov renewal processes, see Remark~\ref{cltr} below for more details. In particular, both the limiting variances can be compactly
expressed in terms of a pseudo inverse (the inverse in a $(k-1)$-dimensional subspace) of a $k\times k$
matrix (which is tridiagonal for $\calj_k(n)$), see, for instance, Eq.~(2) in \cite{clt7}.  
\par
In contrast to $\cals_k(n)$ and $\calj_k(n),$ the distribution of $\cali_k(n)$ converges for all $k\geq 3$ to a proper limit, as $n \to \infty,$ without scaling:
\beqn
\label{this}
&&
\lim_{n\to\infty} P\big(\cali_k(n)\geq r\big)=\frac{h_{0,k}(r)}{k^r}
\feqn
for any $r\in\nn.$ Indeed, for any $n>r$ we have the identity $P\big(\cali_k(n)\geq r\big)=\frac{h_{0,k}(r)}{k^r}.$
The limit is a proper distribution since $\lim_{r\to\infty}\frac{h_{0,k}(r)}{k^r}=0$ (for instance, by the result in part (i) of Theorem~\ref{thpole1} below). Combining \eqref{this} with the explicit formula for $h_{0,k}(r)$ given by Theorem~2.4 in \cite{mansour2}, we arrive to the following:
\begin{theorem}
\label{ilimit}
For all $k\geq 3$  and $r\geq 1,$
\begin{align*}
&
\lim_{n\to\infty} P\big(\cali_k(n)=r\big)=\frac{h_{0,k}(r)}{k^r}-\frac{h_{0,k}(r+1)}{k^{r+1}}
\\
&
\quad
=\frac{2}{(k+1)k^{r+1}}\sum_{j=1}^k \big(1+(-1)^{j+1}\big)\cot^2\frac{j\pi}{2(k+1)}\Big(1+
2\cos \frac{j\pi}{k+1}\Big)^{r-1}\Big(k-1-2\cos \frac{j\pi}{k+1}\Big).
\end{align*}
\end{theorem}
Various additional aspects of the asymptotic behavior of the sequences $\cals_k(n),$ $\cali_k(n),$ and $\calj_k(n)$ can be analyzed
by exploiting the underlying Markovian structure of staircase patterns which is described in Sections~\ref{mac3} and~\ref{mac1}, and
Remark~\ref{hbor} and the standard machinery of the asymptotic theory of Markov renewal processes (cf. Remark~2.4 in \cite{houdre}).
For instance, one can supplement the laws of large numbers in Theorem~\ref{Lstair-llaws} with
exponential concentration inequalities and large deviation estimates (cf. Section~3 in \cite{ldpr}) and the central limit theorems in Theorem~\ref{Lstair-llaws} with the associated laws of the iterated logarithm (cf., for instance, Theorem~1-(iv) in \cite{agut}).
Another natural example is the following ``cousin" of Theorem~\ref{Lstair-avg}-(i).
Let $\pi^{(1)}\pi^{(2)}\cdots\pi^{(\cals_k(\pi)+1)}$ be the representation of a random word $\pi\in[k]^n$ as a
concatenation of maximal staircase subwords $\pi^{(i)},$ $i=1,\ldots,\cals_k(\pi)+1.$ Let $\ell_i$ be the length (the number of letters) of
$\pi^{(i)}.$ Then, with probability one, the following holds for any $k\geq 3:$
\beq
\lim_{n\to\infty} \frac{1}{\cals_k(n)}\sum_{i=1}^{\cals_k(n)}\ell_i=
\lim_{n\to\infty} \frac{1}{\cals_k(n)}\sum_{i=1}^{\cals_k(n)+1}\ell_i=\frac{k^2}{(k-1)(k-2)}.
\feq
In words, $\frac{k^2}{(k-1)(k-2)}$ is the asymptotic average distance between two separations in a random word.
The limit result is immediate from \eqref{mr} and \eqref{apsi} given below, in Section~\ref{mac3}.
Yet another natural extension of asymptotic results in Theorems~\ref{Lstair-avg} and~\ref{Lstair-llaws} is the following proposition
whose short proof is deferred to the Appendix.
\begin{proposition}
\label{prop}
\item [(i)] Let $\pi_1\cdots \pi_{\cali_k(n)}$ be the longest $L$-staircase subword of a random word $\pi\in[k]^n.$
For $w\in [k],$ let $\xi_{w,k}(n)=\#\big\{i \in [\cali_k(n)]\,:\, \pi_i=w\big\}$ be the number of occurrences of the letter
$w$ in the subword. Then, with probability one, the following holds for all $k\geq 3$ and $w\in [k]:$
\beqn
\label{xi}
\lim_{n\to\infty} E\big(\xi_{w,k}(n)\big)=
\left\{
\begin{array}{ll}
6w(4-w)&~\mbox{\rm if}~k=3,
\\
[6pt]
\frac{k}{k-3}\cdot\frac{U_k\big(\frac{k-1}{2}\big)-U_{w-1}\big(\frac{k-1}{2}\big)-U_{k-w}\big(\frac{k-1}{2}\big)}{U_k\big(\frac{k-1}{2}\big)}
&~\mbox{\rm if}~k\geq 4.
\end{array}
\right.
\feqn
\item [(ii)]Let $\pi_1,\ldots, \pi_{\cali_k(n)}$ be the longest $L$-staircase subsequence in a random word $\pi\in[k]^n.$
For $w\in [k],$ let $\eta_{w,k}(n)=\#\big\{i \in [\calj_k(n)]\,:\, \pi_i=w\big\}$ be the number of occurrences of the letter
$w$ in the subsequence. Then, with probability one, the following holds for any $k\geq 3:$
\beqn
\label{eta}
\lim_{n\to\infty} \frac{\eta_{w,k}(n)\big)}{\calj_k(n)}=
\left\{
\begin{array}{ll}
\frac{2}{3k-2}&~\mbox{\rm if}~w\in\{1,k\},
\\
[4pt]
\frac{3}{3k-2}&~\mbox{\rm if}~w\not\in\{1,k\}.
\end{array}
\right.
\feqn
\end{proposition}
Next, we study the asymptotic behavior, as $r\in\nn$ remains fixed and $n$ goes to infinity, of the frequency sequences $h_{r,k}(n)=k^n P\big(\cals_k(n)=r\big),$ $f_{r,k}(n)=k^n P\big(\calj_k(n)=r\big),$ and $g_{r,k}(n)=k^n P\big(\cali_k(n)=r\big).$ We have:
	
\begin{theorem}
\label{thpole1}
The following holds for any $k\geq 3:$
\item [(i)] $\lim_{n\to \infty}\big(h_{r,k}( n)\big)^{\frac{1}{n}} = 1+2\cos(\frac{\pi}{k+1})$ for all integers $r\geq 0.$
\item [(ii)] $\lim_{n\to \infty}\big(f_{r,k}( n)\big)^{\frac{1}{n}} = k-2$ for all integers $r> 0.$
\end{theorem}
The proof of part (i) relies on a result of \cite{mansour2}, it is given in Section~\ref{mac4}.
A short proof of part (ii) is included in Section~\ref{prf}. As regards $g_{r,k}(n),$ we note that Theorem~\ref{ilimit} yields:
\beq
\lim_{n\to \infty}\big(g_{r,k}( n)\big)^{\frac{1}{n}} =\lim_{n\to \infty}\big(k^n\cali_k( n)\big)^{\frac{1}{n}}= k.
\feq
for all $r \in \nn.$
\par
Observe that when $n$ is fixed and $k$ is taken to infinity, $\cals_k(n)$ converges to $n$ while $\calj_k(n)$ converges to $1$
in probability. Theorem~\ref{thpole1} can serve to illustrate the inherent difference between the random mechanism of $\calj_k(\cdot)$ and $\cals_k(\cdot)$ through
a sensible quantification of the dependence of their distribution on the size of the alphabet. Specifically, let
\beq
H^{(s)}_k(n)=-\sum_{r\geq 0} P^{(s)}_{k,n}(r)\log P^{(s)}_{k,n}(r)
\feq
be the entropy of $\cals_k(n).$ We similarly denote by $H^{(lw)}_k(n)$ the entropy of $\cali_k(n)$ and by $H^{(ls)}_k(n)$ the entropy of
$\calj_k(n).$ Loosely speaking, the entropy measures the amount of uncertainty in the value of a random variable \cite{ebook}.
Therefore, one should expect that the entropy sequence $H^{(s)}_k(\cdot)$ is subadditive, namely the following holds for all $n,m\in\nn:$
\beqn
\label{suba}
H^{(s)}_k(n+m)\leq H^{(s)}_k(n) + H^{(s)}_k(m).
\feqn
It is not hard to verify that this heuristic inequality indeed holds. The convergence of $\frac{H^{(s)}_k(n)}{n}$ is thus ensured by Fekete's subadditivity lemma.  Our next result specifies the value of the limit. The proof given in Section~\ref{mac4} however bypasses \eqref{suba} entirely and shows the existence of the limit directly, relying on the results in Theorem~\ref{thpole1}.

\begin{corollary} \label{ent_lls}
For $k\geq 3,$
\item[(i)] $\lim_{n\to \infty} \frac{H^{(s)}_k(n)}{n} = \log {\frac{k}{1+2\cos(\frac{\pi}{k+1})}}.$
\item[(ii)] $\lim_{n\to \infty} \frac{H^{(ls)}_k(n)}{n} = \log {\frac{k}{k-2}}.$
\end{corollary}
The corollary is a direct implication of Theorem~\ref{thpole1}, and proofs of part~(i) and part~(ii) are nearly identical. We thus only prove the former in Section~\ref{mac4} and omit the proof of the latter. We remark that Theorem~\ref{ilimit} implies that $H^{(lw)}_k(n)$ converges, as $n\to\infty,$ to
the entropy of the limiting distribution of the sequence $\cali_k(n)$
(see, for instance, Proposition~1 in \cite{silva} - some care is required because, in general,
the entropy is not a continuous function of distributions in infinite spaces, even in a discrete setting \cite{ebook, silva}).
\par
The rest of the paper is organized as follows. The proof of $d$-separation results (namely, Theorem~\ref{Lstair-avg}-(i),(ii), 
Theorem~\ref{Lstair-llaws}-(i),(ii),  Theorem \ref{thpole1}-(i) and Corollary~\ref{ent_lls}-(i)) are given in Section \ref{ess}.
The results for $L$-staircase subwords (namely, Theorem~\ref{Lstair-avg}-(iii),(iv)) 
are proven in Section~\ref{lolik} and for $L$-staircase subsequences (namely, Theorem~\ref{Lstair-avg}-(v),(vi),
Theorem~\ref{Lstair-llaws}-(iii),(iv), and  Theorem \ref{thpole1}-(ii)) in Section~\ref{ls}.
Finally, the proof of Proposition~\ref{prop} is included in the Appendix. Since the sections are independent of each other and practically self-contained, with slight abuse of notation, we use the same generic notation $L_{n,k}(q)$ and $L_{n,k;l_1l_2\cdots l_s}(q)$ for underlying generating functions throughout the rest of the paper. The actual definitions are different in each subsection, their are always specified at the beginning of a subsection and remain unchanged until its end. In addition, if $f(x)=\sum_{n\geq 0}a_nx^n,$ we frequently use the notation $[x^n]f(x)$ to denote $a_n,$ the $n$-th coefficient of the power series. For example, $[x^2]\cos(x)=-\frac{1}{2}$ and  $[x^3]\sin(x)=\frac{1}{3}.$
\section{Staircases separation}\label{ess}
This goal of this section is to prove the results for the number of separations. The section is divided
into three subsections. Section~\ref{aproof} contains the proof of part (i) and part (ii) of Theorem~\ref{Lstair-avg},
Section~\ref{mac3} includes the proof of part (i) and part (ii) of Theorem~\ref{Lstair-llaws},
and Section~\ref{prf} is devoted to the proof of Theorem \ref{thpole1}-(i) and its implication, Corollary~\ref{ent_lls}-(i).

\subsection{Proof of Theorem~\ref{Lstair-avg}-(i),(ii)}
\label{aproof}
Let $L_{n,k}(q)$ be the generating function of the number of $k$-ary words of length $n,$ according to the statistic $\cals_k.$ That is,
\beq
L_{n,k}(q)=\sum_{\pi\in [k]^n}q^{\cals_k(\pi)}=\sum_{r=0}^{n-1} h_{r,k}(n)q^r, \qquad q\in\cc.
\feq
In addition, we define $L_{n,k;l_1l_2\cdots l_s}(q)$ to be the generating function of the number of $k$-ary words $\pi= l_1l_2\cdots l_s\pi'$
of length $n$ according to the statistic $\cals_k.$ That is,
\beq
L_{n,k;l_1l_2\cdots l_s}(q)=\sum_{\pi=l_1l_2\cdots l_s\pi'\in [k]^n}q^{\cals_k(\pi)}, \qquad q\in\cc.
\feq
Next, for suitable $q,x\in\cc,$ we set
\beqn
\label{Lk}
L_{k}(x,q)=\sum_{n\geq0}L_{n,k}(q)x^n
\qquad
\mbox{\rm and}
\qquad
L_{k;l_1l_2\cdots l_s}(x,q)=\sum_{n\geq0}L_{n,k;l_1l_2\cdots l_s}(q)x^n.
\feqn
Note that the series in \eqref{Lk} converge at least when $k\cdot |x|\cdot \max\{1,|q|\}<1.$
It is straightforward to verify that
\beq
L_1(x,q)=\frac{1}{1-x} \quad \text{and} \quad L_2(x,q)=\frac{1}{1-2x}.
\feq
Hence, in what follows we will assume that $k\geq3.$  By \eqref{Lk}, we have
\beq
L_{k;i}(x,q)&=&x+x(L_{k;i-1}(x,q)+L_{k;i}(x,q)+L_{k;i+1}(x,q))+xq\sum_{j\in [k]\backslash \{i-1,i,i+1\}} L_{k;j}(x,q)
\nonumber
\\
&=&
x+x(L_{k;i-1}(x,q)+L_{k;i}(x,q)+L_{k;i+1}(x,q))
\nonumber
\\
&&
\qquad
+xq(L_{k}(x,q)-1-L_{k;i-1}(x,q)+L_{k;i}(x,q)+L_{k;i+1}(x,q))
\nonumber
\\	
&=&
x(1-q)+x(1-q)(L_{k;i-1}(x,q)+L_{k;i}(x,q)+L_{k;i+1}(x,q))+xqL_k(x,q),
\feq
where
\beq
L_{k;0}(x,q)=L_{k,k+1}(x,q)=0.
\feq
Therefore, for all $i=1,2,\ldots,k,$ we have
\begin{align}
\label{eqM1}
L_{k;i}(x,q)&=\alpha(L_{k;i-1}(x,q)+L_{k;i+1}(x,q))+\beta,
\end{align}
where
\beq
\alpha=\frac{x(1-q)}{1-x(1-q)} \quad \text{and} \quad  \beta=\alpha+\frac{xq}{1-x(1-q)}L_k(x,q).
\feq
\par
Define a symmetric tridiagonal matrix ${\bf A}_k=(a_{ij})_{1\leq i,j\leq k}$ by setting
\beq
a_{ij}=\left\{
\begin{array}{ll}
1&~\mbox{\rm if}~i=j,\\
-\alpha&~\mbox{\rm if}~|i-j|=1,\\
0&~\mbox{otherwise}.
\end{array}
\right.
\feq
In this notation, a  matrix form of \eqref{eqM1} is
\begin{align}
\label{eqM0}
{\bf A}_k
\left(
\begin{array}{l}
L_{k;1}(x,q)
\\
\vdots
\\
L_{k;k}(x,q)\end{array}\right)=\beta e_k,
\end{align}
where $e_k=(1,\ldots,1)\in\rr^k$ is a $k$-dimensional vector with all coordinates equal to one. Finally, taking in account that
\beq
L_k(x,q)=1+\sum_{i=1}^k L_{k;i}(x,q),
\feq
we obtain that 
\beqn
\label{lea}
L_k(x,q)=1+\beta e_k^{\text {\tiny T}}{\bf A}_k^{-1}e_k=1+\alpha e_k^{\text {\tiny T}}{\bf A}_k^{-1}e_k+
\frac{xq}{1-x(1-q)}L_k(x,q)e_k^{\text {\tiny T}}{\bf A}_k^{-1}e_k.
\feqn
An application of Corollary~4.4 in \cite{isumms} then yields the following: 
\begin{lemma}
\label{thmm1a2}
For all $k\geq1,$
\begin{align*}
L_k(x,q)&=1-\frac{1+(k(t-1)-1)U_k(t)+U_{k-1}(t)}{q+2(q-1)(t-1)^2U_k(t) +q(k(t-1)-1)U_k(t)+qU_{k-1}(t)},
\end{align*}
where $t=\frac{1-x(1-q)}{2x(1-q)},$ as long as $t\not\in\{-1,1\}.$ 
\end{lemma}

\begin{example}
\label{eqkk3}
Lemma~\ref{thmm1a2} with $k=3$ gives
\begin{align*}
L_3(x,q)&=\frac{1+x-qx}{1-(2+q)x-(1-q)x^2}
\\
&
=\frac{1+x}{1-2x-x^2}+\sum_{r\geq1}\frac{2x^{r+1}(1-x)^{r-1}}{(1-2x-x^2)^{r+1}}q^r.
\end{align*}
\end{example}

Assume now that $k\geq4.$ Setting $\veps=1-q$ and expanding the numerator and denominator of
the generating function $L_k(x,1-\veps)$ at $\veps=0$ with the help of the identity
\begin{align}
\label{cheb1}
U_n(x)=\sum_{j\geq0}(-1)^j\binom{n-j}{j}(2x)^{n-2j},
\end{align}
we see that 
\beqn
\label{Leps}
L_k(x,1-\veps)=\frac{\sum_{j=0}^2\alpha_jx^j\veps^{j}+O(\veps^3)}{\sum_{j=0}^2\beta_jx^j\veps^j+O(\veps^3)},
\feqn
where
\begin{align*}
\alpha_0&=k+2,\qquad \alpha_1=-(k+3)(k+4),\qquad \alpha_2=\frac{1}{2}(k+3)(k^2+6k+12),\\
\beta_0&=k+2+6\veps,\qquad
\beta_1=-(k^2+7k+12)+\frac{1}{2}(k^2+k-18)\veps,\\
\beta_2&=\frac{1}{2}(k+3)(k^2+6k+12)-\frac{1}{3}k(k^2+6k+11)\veps.
\end{align*}
It is not hard to verify that \eqref{Leps} is equivalent to the following result:
	
\begin{lemma}
\label{Lexpansion}
The power series expansion of the generating function $L_k(x,q)$ at $q=1$ is given by
\beq
L_k(x,q) &=& \frac{a_k(x)}{1-kx}+b_k(x)\frac{x^2(q-1)}{(1-kx)^2}
+c_k(x)\frac{x^3(q-1)^2}{(1-kx)^3}+O((q-1)^3),
\feq
where
\beq
a_k(x)&=&1,\qquad b_k(x)=(k-2)(k-1),\qquad  c_k(x)=-(k-2)(2x-k^2+4k-5).
\feq
\end{lemma}	
Using Lemma~\ref{Lexpansion} to evaluate $\frac{\partial^j}{\partial q^j}L_k(x,q)\mid_{q=1}$ for $j=1,2,$ we obtain that
\begin{align*}
E\big[\cals_k(n)\big] &= \frac{1}{k^n}[x^n]\frac{\partial}{\partial q}L_k(x,q)
\Big|_{q=1}= \frac{1}{k^n}[x^n]\frac{(k-2)(k-1)x^2}{(1-kx)^2}
\\
&=\frac{(k-2)(k-1)}{k^2}(n-1),
\\
E\big[\cals_k(n)(\cals_k(n)-1)\big] &= \frac{1}{k^n}[x^n]\frac{\partial^2}{\partial q^2}L_k(x,q)\Big|_{q=1}  = \frac{1}{k^n}[x^n]
\frac{2(k-2)x^3(k^2-4k+5-2x)}{(1-kx)^3}
\\
&=\frac{(n-2)(k-2)}{k^4}\big((n-1)k^3-4(n-1)k^2+5(n-1)k-2n+6\big).
\end{align*}
Thus,
\begin{align*}
\sigma^2\big(\cals_k(n)\big)&=E\big(\cals_k(n)(\cals_k(n)-1)\big)+E\big(\cals_k(n)\big)-\big[E\big(\cals_k(n)\big)\big]^2
\\
&= -\frac{4(k-2)}{k^4}+\frac{(k-2)(3k^2-5k+6)}{k^4}(n-1).
\end{align*}
This completes the proof of (i) and (ii) of Theorem \ref{Lstair-avg}.
\hfill\hfill\qed

\subsection{Proof of Theorem~\ref{Lstair-llaws}-(i),(ii)}
\label{mac3}	
Recall \eqref{random}. Without loss of generality we can assume that $\cals_k(n)=\cals_k(W_n).$
Thus, there exists a random sequence of \textit{maximal} $L$-staircase words $\big(\pi^{(n)}\big)_{n\in\nn}$
such that
\beq
W_n=\pi^{(1)}\cdots\pi^{(\cals_k(n))}\mu_n,
\feq
where the remainder $\mu_n$ is a prefix of $\pi^{(\cals_k(n)+1)}.$ Denote
by $u_n$ the last letter in the word $\pi^{(n)}$ and by $\ell_n$ the length of $\pi^{(n)}.$
For $n\in\nn,$ let
\beq
\Lambda_n=\sum_{i=1}^n \ell_i.
\feq
Thus, $\cals_k(n)$ can be characterized as the unique sequence of non-negative integers such that
\beqn
\label{mr}
\Lambda_{\cals_k(n)}<n\leq \Lambda_{\cals_k(n)+1}\qquad \forall\,n\in\nn.
\feqn
Equivalently, for any $n,m\in\nn$ we have
\beqn
\label{e3}
\{\cals_k(n)\geq m\}=\{\Lambda_m<n\}.
\feqn
Clearly, both $(u_i)_{i\in\nn}$ and $(u_i,\ell_i)_{i\in\nn}$
are Markov chains, the former is irreducible and aperiodic,
and the latter is formally a \emph{hidden Markov model} (in that the transition kernel of $(u_i,\ell_i)$ depends
only on the current value of the first component, the value of $(u_{i+1},\ell_{i+1})$ is independent of $\ell_i$ when $u_i$ is given).
Furthermore, the distribution tails of $\ell_i$ are exponential, namely
\beq
P(\ell_{i+1}\geq x\mid u_i=j)\leq \Big(\frac{k-2}{k}\Big)^{x-1}\qquad \forall\,i,x\in\nn,j\in[k].
\feq
Therefore \cite{xiac,bmc-book,clt}, there exists a strictly positive number $b>0$ such that  the distribution of the normalized sequence
$\frac{\Lambda_n-an}{b\sqrt{n}}$ converges, as $n\to\infty,$ to $\caln(0,1),$ where $a=\lim_{n\to\infty}\frac{\Lambda_n}{n}.$
Note that by the ergodic theorem, $a=E_{st}(\Lambda_1),$ where $E_{st}$ stands for the expectation in the stationary regime.
\par
The two-component random process $(u_n,\Lambda_n)_{n\in\nn}$ belongs to the class of \emph{Markov renewal processes} \cite{renewal5, renewal}.
Usual arguments of the renewal theory (see, for instance, (4.2) in \cite{agut}, Lemma~2.1.17 in \cite{notes}, or Section~11 in \cite{renewal}) show that \eqref{mr} implies that with probability one,
\beqn
\label{apsi}
a=\lim_{n\to\infty} \frac{n}{\cals_k(n)}=\lim_{n\to\infty} \frac{n}{E(\cals_k(n))}=
\frac{k^2}{(k-2)(k-1)},
\feqn
where the first identity is a direct implication of \eqref{mr}, the second one is an application of the bounded convergence theorem
to the sequence $\cals_k(n)/n,$ and the last one is due to the result in Theorem~\ref{Lstair-avg}-(i).
\par
The translation of the CLT for $\Lambda_n$ into a CLT for $\cals_k(n)$ is a standard duality (locations of the random walk versus first passage times) argument in the renewal theory of random walks on $\zz$ (which is especially simple in our case because the increments $\ell_n$ of the random walk $\Lambda_n$ are strictly positive), see, for instance, \cite[p.~344]{alili}, Theorem~1 in \cite{agut}, or Section~11 in \cite{renewal}.
Specifically, choose any $x\in \rr$ and, for $n\in\nn,$ let
\beq
\tau(n)=\lfloor a^{-1} n+a^{-3/2} b x \sqrt{n} \rfloor,
\feq
where $\lfloor  y \rfloor $ stands for $\max\{n\in\nn\mid n\leq y\},$ the integer part of $y.$ Note that $\tau(n)>0$ for all $n$ large enough.
It follows then from \eqref{e3} that, as $n\to\infty,$
\begin{align*}
P\Big(\frac{\cals_k(n)-a^{-1} n}{a^{-3/2}b \sqrt{n}}\geq x\Big)&=
P\big(\cals_k(n)\geq \tau(n)\big)=
P\big(\Lambda_{\tau(n)}\leq n \big)
\\
&
=
P\bigg(\frac{\Lambda_{\tau(n)}-(n+a^{-1/2}b x \sqrt{n})}{ b\sqrt{a^{-1} n+a^{-3/2}b x \sqrt{n}}}\leq \frac{n-(n+a^{-1/2}b x \sqrt{n})}{b\sqrt{a^{-1} n+a^{-3/2}b x \sqrt{n}}} \bigg)
\\
&
\sim \Phi(-x)=1-\Phi(x),
\end{align*}
where $\Phi(x)=\frac{1}{\sqrt{2\pi}}\int_{-\infty}^xe^{-\frac{x^2}{2}}\,dx $ is the distribution function of $\caln(0,1).$  Thus,
\beq
\lim_{n\to \infty} \call \bigg( \frac{\cals_k(n)-a^{-1} n}{a^{-3/2}b \sqrt{n}} \bigg) = \caln(0,1).
\feq
The proof of part (i) of the theorem is complete.
\hfill\hfill\qed
\begin{remark}
\label{cltr}
The proof shows that $\sigma_k=a^{-3/2}b.$ The limiting variance $b^2$ can be expressed in terms of the underlying Markov chain $(u_n,\ell_n)_{n\in\nn}$ as follows:
\beq
b^2=VAR_{st}(\ell_1)+2\sum_{j=2}^\infty COV_{st}(\ell_1,\ell_j),
\feq
where the subscript ``$st$" indicates that the covariances are calculated in the stationary regime.
For more information on the structure of the limiting variance and its alternative representations
see, for instance, Section~17.4.3 in \cite{bmc-book} and \cite{clt7, clt5}.
\end{remark}
\subsection{Proof of Theorem \ref{thpole1}-(i) and Corollary~\ref{ent_lls}-(i)}
\label{mac4}
We begin with the proof of Theorem~\ref{thpole1}-(i).
\\
$\mbox{}$
\\
\textit{Proof of Theorem \ref{thpole1}-(i).}
By Corollary~2.5 in \cite{mansour2}, we know that the result holds for $r=0,$ that is
\beqn
\label{hbo}
\lim_{n\to\infty} (h_{0,k}(n))^{1/n}=1+2\cos\Big(\frac{\pi}{k+1}\Big).
\feqn
Therefore, in order to complete the proof for general $r\geq 0$ it suffices to show the following:
\begin{proposition}
\label{procr2}
For all $r\geq0$ and $k\geq 3,$ $\lim_{n\to\infty} (h_{r,k}(n))^{1/n}$ exists and its value is independent of $r.$
\end{proposition}
\begin{proof}
By Example \ref{eqkk3}, the proposition holds for $k=3.$ Thus, from now on we will assume that $k\geq4.$
We proceed with the proof by induction on $r.$ Denote
\beq
\xi_k=1+2\cos\Big(\frac{\pi}{k+1}\Big).
\feq
Assume now that $\lim_{n\to\infty} (h_{r-1,k}(n))^{1/n}=\xi_k$ for some $r\in\nn.$
For $n\in\nn$ and integer $t \geq 0,$ define
\beq
\Theta_{t,k}(n) = \big\{ \pi \in [k]^n \,:\, \cals_k(\pi) = t \big\}.
\feq
Let $w_1w_2\cdots w_n$ be any word in the set $\Theta_{r-1,k}(n)$ and choose a letter $v\in[k]$ such that $|v-w_n|>1.$ Then
$h_{r-1,k}(n)\leq h_{r,k}(n+1),$ which implies
\beq
\xi_k=\lim_n(h_{r-1,k}(n))^{1/n}\leq \liminf_n (h_{r,k}(n))^{1/n}.
\feq
Let now $w=w_1w_2\cdots w_n\in \Theta_{r,k}(n)$ and $i$ be minimal index such that $|w_i-w_{i+1}|>1.$ Then
$w_1w_2\cdots w_i\in \Theta_{0,k}(i)$ and $w_{i+1}w_{i+2}\cdots w_n\in \Theta_{r-1,k}(n-i).$ Thus,
\begin{align}\label{eqhrfi}
h_{r,k}(n)\leq \sum_{i=0}^n h_{0,k}(i)h_{r-1,k}(n-i).
\end{align}
By \eqref{hbo} and the induction hypothesis, for every $\veps>0$ there exists a constant $A_\veps>0$ such that for all $n\in \nn$ we have
\beq
h_{0,k}(n) \leq A_\veps(\xi_k+\veps)^n\quad \mbox{\rm and }\quad h_{r-1,k}(n)\leq A_\veps(\xi_k+\veps)^n,
\feq
which, by virtue of \eqref{eqhrfi}, leads to
\beq
h_{r,k}(n)\leq nA_\veps^2 (\xi_k+\veps)^n.
\feq
Since $\veps>0$ is an arbitrary constant, this implies that $\limsup_{n\to\infty}\big(h_{r,k}(n)\big)^{1/n}\leq \xi_k.$ Combining the liminf and limsup inequalities we get
\beq
\xi_k\leq\liminf_n(h_{r,k}(n))^{1/n}\leq\limsup_n(h_{r,k}(n))^{1/n}\leq\xi_k,
\feq
which implies the claim of the proposition.
\end{proof}
\begin{remark}
\label{hbor}
The key limit identity \eqref{hbo} is obtained in \cite{mansour2} by analytical methods. We will now outline a short probabilistic argument
leading to this result (note however that Corollary~2.5 in \cite{mansour2} gives more information about the asymptotic behavior of $h_{0,k}(n)$ than
the limit in \eqref{hbo} alone). Consider a discrete-time Markov chain $(Y_\ell)_{\ell\geq 0}$ on the state space $[k]\cup\{0\}$ with transition kernel
\beq
P_{ij}:=P(Y_{\ell+1}=j\,|\,Y_\ell=i)=
\left\{
\begin{array}{ll}
\frac{1}{k}&\mbox{\rm if}~|i-j|\leq 1~\mbox{\rm and}~i,j\neq 0,
\\
[4pt]
\frac{k-2}{k}&\mbox{\rm if}~i\in\{1,k\}~\mbox{\rm and}~j=0,
\\
[4pt]
\frac{k-3}{k}&\mbox{\rm if}~i\not\in\{1,k\}~\mbox{\rm and}~j=0,
\\
[4pt]
1&\mbox{\rm if}~i=j=0,
\\
[4pt]
0&\mbox{\rm otherwise}.
\end{array}
\right.
\feq
Thus $0$ is the absorbing state of the Markov chain, which is otherwise irreducible and aperiodic. We will assume that the initial state $Y_0$ is chosen uniformly from
the set $[k].$ In terms of this Markov chain,
\beq
h_{0,k}(n)=k^n P(\cals_{n,k}=0)=k^nP(T_0\geq n),
\feq
where $T_0=\inf\{\ell\in\nn\,:\, Y_\ell=0\}$ is the first hitting time of the absorption state zero. Thus
\beq
h_{0,k}(n)=k^n\cdot \frac{1}{k}e_k^{\text {\tiny T}}Q^{n-1}e_k=k^{n-1}\cdot e_k^{\text {\tiny T}}Q^{n-1}e_k,
\feq
where $Q$ is $k\times k$ matrix which is obtained from $P$ by removing the row and column corresponding to the absorption state zero,
$e_k$ is a $k$-vector with all entries equal to one,
and $\frac{1}{k}$ in the formula stands for the initial distribution.
Therefore, $\lim_{n\to\infty}\big( h_{0,k}(n)\big)^{1/n}=\lambda_k,$ where $\lambda_k$ is the Perron-Frobenius eigenvalue of $kQ.$ It is well-known that $\lambda_k=1+2\cos\big(\frac{\pi}{k+1}\big)$ (see, for instance, \cite{nash} or \cite[p.~239]{rwb}). We remark in passing that $kQ=-\frac{1}{\alpha}({\bf A}_k-I),$ where ${\bf A}_k$ is the matrix that appears in \eqref{eqM0}. This observation provides a direct link between the analytic (generating function) and probabilistic approaches to the proof of Theorem~\ref{thpole1}-(i).
\end{remark}
We now turn to the proof of Corollary~\ref{ent_lls}-(i).
\begin{proof}[Proof of Corollary~\ref{ent_lls}-(i)]
We have
\beq
H^{(s)}_k(n)&=&-\sum_{r\geq 0} P^{(s)}_{k,n}(r)\log P^{(s)}_{k,n}(r)=-\frac{1}{k^n}\sum_{r\geq 0} h_{r,k}(n)\big(\log h_{r,k}(n)-n\log k\big)
\\
&=& n\log k-\sum_{r\geq 0} P^{(s)}_{k,n}(r) \log h_{r,k}(n).
\feq
Thus,
\beq
\frac{H^{(s)}_k(n)}{n} =\log k-\sum_{r\geq 0} P^{(s)}_{k,n}(r)\log \big[ h_{r,k}(n) \big]^{\frac{1}{n}}.
\feq
An application of the bounded convergence theorem (notice that $1\leq h_{r,k}(n)\leq k^n$)
exploiting the result in Theorem~\ref{thpole1}-(i) finishes the proof.
\end{proof}

\section{Longest L-staircase subwords}
\label{lolik}	
The section is devoted to the proof of parts (iii) and (iv) of Theorem~\ref{Lstair-avg}.
\par	
Let $L_{n,k}(q)$ be the generating function for the number of $k$-ary words of length $n$ according to the statistic $\cali_k.$ That is,
\beq
L_{n,k}(q)=\sum_{\pi\in [k]^n}q^{\cali_k(\pi)}=\sum_{r=1}^n g_{r,k}(n)q^n, \qquad q\in\cc.
\feq
In addition, let $L_{n,k;i_1i_2\cdots i_s}(q)$ be the generating function for the number of $k$-ary words $\pi=i_1i_2\cdots i_s\pi'$ of length $n$
according to the statistic $\cali_k,$ that is
\beq
L_{n,k;i_1i_2\cdots i_s}(q)=\sum_{\pi=i_1i_2\cdots i_s\pi'\in [k]^n}q^{\cali_k(\pi)}, \qquad q\in\cc.
\feq
Similarly to \eqref{Lk}, for suitable $q,x\in\cc,$ let
\beq
L_{k}(x,q)=\sum_{n\geq0}L_{n,k}(q)x^n\qquad \mbox{\rm and }\qquad L_{k;i_1i_2\cdots i_s}(x,q)=\sum_{n\geq0}L_{n,k;i_1i_2\cdots i_s}(q)x^n.
\feq
It is easy to verify that $L_1(x,q)=\frac{1}{1-xq}$ and $L_2(x,q)=\frac{1}{1-2xq}.$ Hence, from now on we will assume that $k\geq3.$
We first observe that
\begin{align*}
L_{k;i}(x,q)&=xq+\sum_{j=1}^k L_{k;ij}(x,q)\\
&=xq+xq(L_{k;i-1}(x,q)+L_{k;i}(x,q)+L_{k;i+1}(x,q))+xq\sum_{j\in [k]\backslash \{i-1,i,i+1\}} L_{k;j}(x,1),
\end{align*}
where $L_{k;0}(x,q)$ and $L_{k,k+1}(x,q)$ are defined as $0.$ Since $L_{k;j}(x,1)=\frac{x}{1-kx},$ we obtain that
the following holds for all $i=1,2,\ldots,k:$
\begin{align}
\nonumber
L_{k;i}(x,q)&=\alpha'\big(L_{k;i-1}(x,q)+L_{k;i+1}(x,q)\big)
\\
\label{eqMN1}
&
\qquad
+\alpha'+(k-2)\beta'\delta_{i=1,k}+(k-3)\beta'\delta_{2\leq i\leq k-1},
\end{align}
where
\beq
\alpha'=\frac{xq}{1-xq}, \qquad
\beta'=\frac{x^2q}{(1-xq)(1-kx)},
\feq
and $\delta_X$ for a statement $X$ is defined as $1$ whenever $X$ holds true and $0$ otherwise.
\par 	
Using argument similar to those we employed in the previous section, we can write \eqref{eqMN1} in terms of a symmetric tridiagonal matrix, whose inverse yields
\begin{align}
\nonumber
& L_k(x,q)
\\
&
\qquad
=1+\frac{\alpha'+(k-3)\beta'}{\alpha'U_k(t)}\bigg(\sum_{i=0}^{k-1} U_i(t)U_{k-1-i}(t)
+2\sum_{i=0}^{k-2}\sum_{j=0}^{k-2-i}U_i(t)U_j(t)\bigg)
+\frac{2\beta'}{\alpha'U_k(t)}\sum_{i=0}^{k-1}U_i(t) \nonumber
\\
&
\qquad
= 1+\frac{\alpha'+(k-3)\beta'}{2(t-1)^2\alpha'U_k(t)}
\big(\big(k(t-1)-1\big)U_k(t)+U_{k-1}(t)+1\big) \nonumber
\\
&
\qquad \qquad +\frac{\beta'}{(t-1)\alpha'U_k(t)}\big(U_k(t)-U_{k-1}(t)-1\big). \label{MKxq}
\end{align}
It follows from \eqref{MKxq} that for any $k\geq 3,$
\begin{align*}
&
\frac{\partial}{\partial q}L_k(x,q)\mid_{q=1}
\\
&
\quad
=\frac{2x^2}{(1-kx)(1-3x)^2U_k\big(\frac{1-x}{2x}\big)}
+\frac{2x^2U_{k-1}\big(\frac{1-x}{2x}\big)}{(1-kx)(1-3x)^2U_k\big(\frac{1-x}{2x}\big)} -\frac{x(3kx-k+2x)} {(1-kx)(1-3x)^2}.
\end{align*}
By using the formula $\frac{d}{dx}U_n(x)=\frac{(k+1)T_{k+1}(x)-xU_k(x)}{x^2-1},$
where $T_k(x)$ is the $k$-th Chebyshev polynomial of the first kind, we obtain:
\begin{align*}
&\frac{\partial^2}{\partial q^2}L_k(x,q)\mid_{q=1} =\frac{-2(9kx-3k+6x+2)x^2}{(1-kx)(1-3x)^3}
+
\frac{4x^3(3x+5)}{(1-3x)^3(1-kx)(1+x)U_k\big(\frac{1-x}{2x}\big)}
\\
&
\qquad
\qquad
+\frac{4(1+3x)x^2U_{k-1}\big(\frac{1-x}{2x}\big)}{(1-kx)(1-3x)^2U_k\big(\frac{1-x}{2x}\big)}
+\frac{8(k+1)x^3\big(U_{k-1}\big(\frac{1-x}{2x}\big)+1\big)T_{k+1}
\left(\frac{1-x}{2x}\right)}{(1-kx)(1-3x)^3(1+x)U_k^2\big(\frac{1-x}{2x}\big)}
\\ &
\qquad
\qquad
-\frac{8kx^3T_k\big(\frac{1-x}{2x}\big)}{(1-kx)(1-3x)^3(1+x)U_k\big(\frac{1-x}{2x}\big)}.
\end{align*}
For $k\geq4,$ this implies that $E\big(\cali_k(n)\big)\sim e_1$ and $E\big(\cali_k(n)(\cali_k(n)-1)\big)\sim e_2,$ where
\beqn
\label{e1}
e_1(k)=\frac{2}{(k-3)^2U_k\big(\frac{k-1}{2}\big)}
+\frac{2U_{k-1}\big(\frac{k-1}{2}\big)}{(k-3)^2U_k\big(\frac{k-1}{2}\big)} +\frac{(k^2-3k-2)}{(k-3)^2}
\feqn
and
\begin{align}
\nonumber
e_2(k)&=\frac{-2(6+11k-3k^2)}{(k-3)^3} +\frac{4(3+5k)}{(k-3)^3(k+1)U_k\big(\frac{k-1}{2}\big)}
+\frac{4(k+3)U_{k-1}\big(\frac{k-1}{2}\big)}{(k-3)^2U_k\big(\frac{k-1}{2}\big)}
\\
\label{e2}	
&\qquad +\frac{8k\big(U_{k-1}\big(\frac{k-1}{2}\big)+1\big)T_{k+1}
\big(\frac{k-1}{2}\big)}{(k-3)^3U_k^2\big(\frac{k-1}{2}\big)}
-\frac{8k^2T_k\big(\frac{k-1}{2}\big)}{(k-3)^3(k+1)U_k\big(\frac{k-1}{2}\big)}.
\end{align}
In particular, $\sigma^2\big(\cali_k(n)\big)\sim e_2+e_1-e_1^2.$ The proof is complete. \hfill \hfill\qed

\section{Longest L-staircase subsequences}
\label{ls}
This section is devoted to the proof of the results for longest $L$-staircase subsequences.
Section~\ref{prooft} contains the proof of parts (v) and (vi) of Theorem~\ref{Lstair-avg},
Section~\ref{mac1} includes the proof of parts -(iii) and (iv) of Theorem~\ref{Lstair-llaws},
and the proof part (ii) of Theorem~\ref{thpole1} is given in Section~\ref{prf}.
\subsection{Proof of Theorem~\ref{Lstair-avg}-(v), (vi)}
\label{prooft}
Let $L_{n,k}(q)$ be the generating function of the number of $k$-ary words of length $n,$ according to the statistic $\calj_k.$ That is,
\beq
L_{n,k}(q)=\sum_{\pi\in [k]^n}q^{\calj_k(\pi)}=\sum_{r=1}^n f_{r,k}(n)q^n, \qquad q\in\cc.
\feq
In addition, we define $L_{n,k;l_1l_2\cdots l_s}(q)$ to be the generating function of the number of $k$-ary words $\pi= l_1l_2\cdots l_s\pi'$ of length $n$ according to the statistic $\calj_k.$ That is,
\beq
L_{n,k;l_1l_2\cdots l_s}(q)=\sum_{\pi=l_1l_2\cdots l_s\pi'\in [k]^n}q^{\calj_k(\pi)}, \qquad q\in\cc.
\feq
Next, similarly to \eqref{Lk}, for suitable $q,x\in\cc,$ we set
\beqn \label{llLk}
L_{k}(x,q)=\sum_{n\geq0}L_{n,k}(q)x^n
\qquad\mbox{\rm and} \qquad
L_{k;l_1l_2\cdots l_s}(x,q)=\sum_{n\geq0}L_{n,k;l_1l_2\cdots l_s}(q)x^n.
\feqn
It is straightforward to verify that $L_1(x,q)=\frac{1}{1-xq}$ and $L_2(x,q)=\frac{1}{1-2xq}.$
Hence, in what follows we will assume that $k\geq3.$  By \eqref{llLk}, we have
\beq
L_{k;i}(x,q)&=&xq+xq(L_{k;i-1}(x,q)+L_{k;i}(x,q)+L_{k;i+1}(x,q))+x\sum_{j\in[k]\backslash \{i-1,i,i+1\}}L_{k;i}(x,q)\\ &=&xq+xq(L_{k;i-1}(x,q)+L_{k;i}(x,q)+L_{k;i+1}(x,q))+(k-3)xL_{k;i}(x,q),
\feq
for $i=2,\ldots,k-1,$ and
\beq
L_{k;1}(x,q)&=&xq+xq(L_{k;1}(x,q)+L_{k;2}(x,q))+x\sum_{j\in [k]\backslash \{1,2\}} L_{k;1}(x,q)\\ &=&xq+xq(L_{k;1}(x,q)+L_{k;2}(x,q))+(k-2)xL_{k;1}(x,q),
\\
L_{k;k}(x,q)&=&xq+xq(L_{k;k}(x,q)+L_{k;k-1}(x,q))+x\sum_{j\in [k]\backslash \{k-1,k\}} L_{k;k}(x,q)\\ &=&xq+xq(L_{k;k}(x,q)+L_{k;k-1}(x,q))+(k-2)xL_{k;k}(x,q).
\feq
Let
\beq
\alpha=\frac{xq}{1-(k-2+q)x} \quad \text{and} \quad  \beta=\frac{xq}{1-(k-3+q)x}.
\feq
Suppose that $xq\neq 0,$ so that $\alpha\beta \neq 0.$ 
We define a tridiagonal matrix ${\bf B}_k=(a_{ij})_{1\leq i,j\leq k}$ by setting
\beq
b_{ij}=
\left\{
\begin{array}{ll}
1/\alpha &~\mbox{\rm if}~ i=j=1~\mbox{\rm or}~i=j=k,\\
1/\beta &~\mbox{\rm if}~ i=j~\mbox{\rm and}~ i\not \in\{1,k\},\\
-1&~\mbox{\rm if}~|i-j|=1,\\
0&~\mbox{otherwise}.
\end{array}
\right.
\feq
In the matrix notation, we have
\begin{align}\label{lleqM0}
{\bf B}_k
\left(
\begin{array}{l}L_{k;1}(x,q)
\\
\vdots\\
L_{k;k}(x,q)\end{array}\right)=
e_k,
\end{align}
where $e_k\in\rr^k$ is a $k$-vector with all entries equal to zero. Similarly to \eqref{lea}, 
we obtain that 
\beq
L_k(x,q)=1+\sum_{i=1}^kL_{k,i}(x,q)=
1+e_k^{\text {\tiny T}}{\bf B}_k^{-1}e_k.
\feq
It follows from Corollaries~4.2-4.4 in \cite{isumms} (apply first Corollary~4.4 in \cite{isumms} to the $(k-2)\times (k-2)$ 
matrix obtained by deleting the first and the last rows and columns in ${\bf B}$, alternatively we could use Theorem~2 in \cite{tan}) 
that 
\beqn
\label{new}
L_k(x,q)&=&\frac{\beta\Big(k-2-\frac{2\beta\big(1-\frac{1+U_{k-3}(1/2\beta)}{U_{k-2}(1/2\beta)}\big)}{1-2\beta}\Big)}{1-2\beta}+
2\alpha\Big(1-\frac{U_{k-3}(1/2\beta)}{U_{k-2}(1/2\beta)}\Big).
\feqn  
Differentiating $L_k(x,q)$ with respect to $q$ at $q=1$ we obtain that
\beq
\frac{\partial}{\partial q}L_k(x,q) \Big|_{q=1}=\frac{x(k-(k-2)(k-1)x)}{(1-kx)^2},
\feq
which leads to $E(\calj_k(n))=1+\frac{3k-2}{k^2}(n-1).$
Also, by differentiating $L_k(x,q)$ twice with respect to $q$ and then substituting $q=1,$ we obtain that there exist two constants $a_k$ and $b_k$ and a function $\witi L_k(x)$ analytic at $x=1/k$ such that
\beq
\frac{\partial^2}{\partial q^2}L_k(x,q)\Big|_{q=1}=\frac{2(3k-2)^2}{k^4(1-kx)^3}+\frac{a_k}{(1-kx)^2}+\frac{b_k}{1-kx}+\witi L_k(x),
\feq
where
\beqn
a_k&=&\lim_{x\to 1/k}\bigg(\frac{\partial^2}{\partial q^2}L_k(x,q)\Big|_{q=1}-\frac{2(3k-2)^2}{k^4(1-kx)^3}\bigg)(1-kx)^2,
\label{a_n_lab}
\\
b_k&=&\lim_{x\to 1/k}\bigg(\frac{\partial^2}{\partial q^2}L_k(x,q)\Big|_{q=1}-\frac{2(3k-2)^2}{k^4(1-kx)^3}-\frac{a_k}{(1-kx)^2}\bigg)(1-kx).
\nonumber
\feqn
Thus, by using the method of residuals \cite{analcombin} to calculate the coefficient in front of $x^2,$ we obtain that
\beq
E\big(\calj_k(n)(\calj_k(n)-1)\big)=\frac{(3k-2)^2(n+1)(n+2)}{k^4}+a_k(n+1)+O(1),
\feq
and
\beq
\sigma^2\big(\calj_k(n)\big)&=&
E\big(\calj_k(n)(\calj_k(n)-1)\big)+E(\calj_k(n))-E\big[\big(\calj_k(n)\big)\big]^2
\\
&=&
\Big(a_k-\frac{(3k-2)(k^2-15k+10)}{k^4}\Big)n+O(1).
\feq
The proof is complete. \hfill\hfill\qed
\subsection{Proof of Theorem~\ref{Lstair-llaws}-(iii),(iv)}
\label{mac1}
Recall \eqref{random}. Without loss of generality we can assume that $\calj_k(n)=\calj_k(W_n).$
Let $T(1)=1$ and, recursively, for $\ell\in\nn,$
\beq
T(\ell+1)=\inf\{t\in\nn\,:\,|w_t-w_{T(\ell)}|\leq1\}.
\feq
Denote $X_\ell=w_{T(\ell)}.$ Then, for all $n\in\nn,$  $X_1\cdots X_{\calj_k(n)}$ is the longest $L$-staircase subsequence of $W_n,$
and $\calj_k(n)$ is the unique integer such that
\beqn
\label{tn}
T\big(\calj_k(n)\big)\leq n <T\big(\calj_k(n)+1\big).
\feqn
Let $T(0)=0$ and $\tau_n=T_n-T_{n-1}$ for $n\in\nn.$ Then $(X_n)_{n\in\nn}$ is a reflected nearest-neighbor
symmetric random walk on $[k],$ and $(X_n,\tau_n)_{n\in\nn}$ is a Markov chain with transition kernel
\beq
P\big(X_{n+1}=y,\tau_n=t\big|X_n=x\big)=
\left\{
\begin{array}{ll}
\frac{1}{k}\Big(\frac{k-3}{k}\Big)^{t-1}&~\mbox{if}~t\in\nn, x\in\{2,\ldots,k-1\},|y-x|\leq 1,  \\ [8pt]
\frac{1}{k}\Big(\frac{k-2}{k}\Big)^{t-1}&~\mbox{if}~t\in\nn, x\in\{1,k\},|y-x|\leq 1,\\ [8pt]
0&~\mbox{\rm otherwise}.
\end{array}
\right.
\feq
The rest of the proof is similar to the one given in Section~\ref{mac3}:
the Markov renewal process $\big(X_n,T(n)\big)_{n\in\nn}$ plays the role similar to that of $(u_n,\Lambda_n)_{n\in\nn},$
we first proof a CLT for the additive process $T(n),$ and then transform it into a CLT for $\calj_k(n)$ using \eqref{tn}.
In particular, similarly to \eqref{apsi},
\beq
\varphi_k=\lim_{n\to \infty} \frac{\calj_k(n)}{n} =\lim_{n\to \infty} \frac{E(\calj_k(n))}{n}= \frac{3k-2}{k^2},
\feq
where in the last step we used the result in Theorem~\ref{Lstair-avg}-(iv).
\hfill\hfill\qed
\subsection{Proof of Theorem~\ref{thpole1}-(ii)}
\label{prf}
Using \eqref{new}, it is not hard to check that the function $L_k(x,0)=\sum_{n=1}^\infty f_{r,k}(n)x^n$ is meromorphic
and that $1/(k-2)$ is its smallest by absolute value pole.  The result in Theorem~\ref{thpole1}-(ii) follows from this observation \cite{analcombin}.
\par
Alternatively, one can use the following simple counting argument to prove the result. On one hand,
\beq
f_{r,k}(n)\leq \binom{n +r-1}{r-1}\cdot h_{0,k}(r) \cdot (k-2)^{n-r}.
\feq
The first term on the right-hand side of the inequality is the number of ways to locate a given $L$-staircase within a
word of length $n,$ the second is for the number of staircase words of length $r,$ and the third one is an upper bound for the number of ways
to fill the remaining $n-r$ positions.
On the other hand, considering exclusively $L$-staircase subsequences that contain letter $1$ only,
we obtain that
\beq
f_{r,k}(n)\geq\binom{n+r-1}{r-1}\cdot (k-2)^{n-r}.
\feq
The comparison of the lower and upper bounds concludes the proof. \hfill\hfill\qed

\section*{Appendix. Proof of Proposition~\ref{prop}}
(i)~ Recall the random walk $(Y_n)_{n\in\nn},$ the absorption time $T_0,$ and the sub-stochastic transition matrix
$Q$ introduced in Remark~\ref{hbor}. By the monotone convergence theorem we have:
\beq
\lim_{n\to\infty}E\big(\xi_{w,k}(n)\big)=\sum_{m=1}^\infty P(Y_m=w,T_0>m)=\sum_{m=0}^\infty \sum_{j=1}^k\frac{1}{k}Q^m(j,w)=
\frac{1}{k}e(I-Q)^{-1}(w),
\feq
where $e$ is a $k$-vector with all entries equal to one. After a little algebra, the result then follows from Corollary~4.3 in \cite{isumms} (see also Theorem~2 in \cite{tan}). \hfill\hfill\qed
\\
$\mbox{}$
\\
(ii)
Recall the Markov chain $(X_n)_{n\in\nn}$ from Section~\ref{mac1}. Let $\big(P_{ij}\big)_{i,j\in[k]}$
denote transition kernel and $\mu=(\mu_1,\ldots, \mu_k)$ be the unique stationary distribution of the Markov chain.
That is, $\mu_i\geq 0$ for all $i\in [k],$ $\sum_{i\in [k]}\mu_i=1,$ and
$\mu_j=\sum_{i\in [k]}\mu_i P_{ij}$ for all $j\in [k].$ It is easy to verify that
\beq
\mu_1=\mu_k=\frac{2}{3k-2}\qquad \mbox{\rm and}\qquad \mu_2=\cdots=\mu_{k-1}=\frac{3}{3k-2}.
\feq
Since the Markov chain $X_n$ is irreducible and aperiodic, the law of large numbers (ergodic theorem) implies that
\beqn
\label{freq}
\lim_{n\to\infty}\frac{1}{n}\sum_{i=1}^n \one{X_i=w}=\mu_w\qquad \forall\,w\in [k],
\feqn
with probability one, regardless of the initial distribution of the Markov chain. Here $\odin_A$ denotes the indicator of the event $A,$
that is its value is one if the event occurs and is zero otherwise.  To complete the proof,
we note that $\eta_{w,k}(n)=\sum_{i=1}^{\calj_k(n)}\one{X_i=w}$ and pass in \eqref{freq} from the regular limit to the limit along the (random, but diverging with probability one) subsequence $\calj_k(n),$ $n\in\nn.$
\hfill\hfill\qed
\nocite{*}


\end{document}